\documentclass{mpi2015-cscpreprint}

\usepackage[american]{babel}
\usepackage{amssymb}
\usepackage{amsthm,amsmath}
\usepackage{algorithm}%
\usepackage{algorithmicx}%
\usepackage{algpseudocode}%
\usepackage{graphicx}
\usepackage{appendix}
\usepackage{multirow}
\usepackage{array}
\usepackage{url}

\usepackage{subcaption}

\newtheorem{theorem}{Theorem}[section]
\newtheorem{corollary}{Corollary}

\newtheorem{remark}{Remark}[section]

\numberwithin{equation}{section}

\begin{document}

\title{Periodic solutions of nonlinear control systems with switching: a Lie-algebraic and contraction approach}

\author{Alexander~Zuyev$^{1,2}$, Peter~Benner}
\affil[1]{Max Planck Institute for Dynamics of Complex Technical Systems, Magdeburg, Germany
    \email{zuyev@mpi-magdeburg.mpg.de}, \orcid{0000-0002-7610-5621}
    \email{benner@mpi-magdeburg.mpg.de}, \orcid{0000-0003-3362-4103}
}
\affil[2]{Institute of Applied Mathematics and Mechanics, National Academy of Sciences of Ukraine}




\abstract{
This paper is devoted to the analysis of periodic solutions of nonlinear control-affine systems with bang-bang controls. Such problems naturally arise in periodic optimal control with constrained inputs, which have, in particular, important applications in the performance optimization of chemical reactions. We reduce the problem of constructing a periodic solution to that of finding a fixed point of a composition of exponential maps. The latter problem is then addressed using the Baker--Campbell--Hausdorff--Dynkin (BCHD) formula. We establish the equivalence between periodic solutions of the original control system and those of an associated autonomous system involving iterated Lie brackets. Applying {incremental stability arguments} allows us to further simplify the problem to finding the equilibria of this autonomous system. The developed theory is then applied to nonlinear chemical reaction models with constrained controls.
}

 \novelty{
\begin{itemize}
\item The relation between periodic solutions of a nonlinear control-affine system with bang-bang controls and an associated autonomous nonlinear differential equation is established using the Baker–Campbell–Hausdorff–Dynkin (BCHD) formula.
\item This scheme reformulates the original ODE problem as the problem of finding equilibria of an associated autonomous vector field $F$, whose truncated expansion is implemented algorithmically in \texttt{Maple}.
\item Periodic trajectories of controlled nonisothermal chemical reaction models with piecewise-constant input modulations are characterized by combining this approach with incremental stability theory.
\end{itemize}}

\maketitle

\section{INTRODUCTION}\label{sec_intro}

Optimal control problems for nonlinear systems with periodic boundary conditions constitute an important area of research in mathematical control theory~\cite{colonius2006optimal,bittanti2009periodic,zanon2016periodic}, with significant applications in physics~\cite{bai2021quantum,das2023precision}, life sciences~\cite{yuan2020reconciling,ali2021maximizing}, and engineering~\cite{upreti2013optimal}.
{
Our current study is motivated by highly important practical problems of periodic optimization of chemical reactions. It has been known in the chemical engineering literature that periodic operation of essentially nonlinear reactions can improve the mean product yield compared to steady-state operation (see, e.g,~\cite{douglas1967periodic,sinvcic1980analytical,hoffmann1986influence,CES2017} and references therein). For a class of hydrolysis reactions with sinusoidal modulation of input signals, experimental measurements demonstrating such productivity improvement have recently become available~\cite{felischak2021analysis}. However, a rigorous theoretical analysis of these problems is still far from complete.}
In the case of constrained controls, the necessary optimality conditions derived from Pontryagin's maximum principle for this class of problems typically result in bang-bang extremal controls, whose switching times exhibit nontrivial properties due to the coupled structure of the associated Hamiltonian system (cf.~\cite{CES2017}).

This problem was analyzed in our recent paper~\cite{IEEE2025}, where
{a simple iteration and Newton-type methods were developed to construct $\tau$-periodic trajectories of the system $\dot x(t)=Ax(t)+g(x(t))+u(t)$ under the assumptions that the matrix $e^{-\tau A}-I$ is nonsingular and that the nonlinearity $g$ satisfies a Lipschitz condition.
In contrast to the above assumptions, the present work does not require such a dominant linearization assumption and
instead aims to localize periodic solutions of general control-affine systems with nonlinear vector fields.}
The key idea in this direction is to identify fixed points of an associated composition of flows and to reformulate the problem in terms of the flow of a time-invariant system,
whose vector field is defined as the logarithm of a certain product of exponential maps in the driftless representation of the original system.

The subsequent sections are organized as follows.
The periodic control problem for a control-affine system with bang-bang inputs is introduced and reformulated in the sense of a composition of exponential maps in Section~\ref{sec_problem}.
Solvability conditions for this periodic problem are proposed in Section~\ref{sec_results} in terms of the Baker--Campbell--Hausdorff--Dynkin (BCHD) formula. In addition, we apply the {theory of incremental stability~\cite{angeli2002lyapunov,forni2013differential}} to characterize the controlled dynamics in a neighborhood of the considered periodic solutions.
These theoretical results are applied to controlled nonisothermal chemical reaction models with piecewise constant input modulations in Section~\ref{sec_example}.
Conclusions and future perspectives are outlined in Section~\ref{sec_conc}.

\section{PROBLEM FORMULATION AND PRELIMINARIES}\label{sec_problem}
Consider a nonlinear control-affine system of the form
\begin{equation}\label{cs_affine}
\dot x = f_0(x) + \sum_{j=1}^m u_j g_j(x),\quad x \in D\subset {\mathbb R}^n,\; u\in U\subset {\mathbb R}^{{m}},
\end{equation}
where the vector fields $f_0:D\to{\mathbb R}^n$ and $g_j:D\to{\mathbb R}^n$ are assumed to be
{sufficiently smooth}
in the domain $D$.
{The precise regularity assumptions required will be specified in each statement.}
For a given time horizon $\tau>0$, an integer $N\ge 1$, and a finite sequence
$
0=\alpha_0 < \alpha_1 < \alpha_2 < ...< \alpha_N=1
$,
we introduce a partition of the interval $[0,\tau]$ into subintervals
$I_k = [\alpha_{k-1}\tau,\alpha_k\tau)$ for $k=1,..., N-1$, and $I_N = [\alpha_{N-1}\tau,\alpha_N\tau]$.
Let $u_\tau:[0,\tau]\to U$ be a bang-bang control with $N-1$ switchings such that
\begin{equation}\label{u_bangbang}
u_\tau(t) = u^{(k)}={(u^{(k)}_1,...,u^{(k)}_m)^\top} \;\text{for}\; t\in I_k,\;k=1,...,N,
\end{equation}
for some ``switching scenario'' -- a finite sequence of control values $u^{(k)}\in U$, $k=1,2,..., N$.
We denote by $x(t;x^0,u_\tau)$ the solution of system~\eqref{cs_affine} with the initial condition $x(0;x^0,u_\tau)=x^0\in D$ and control $u=u_\tau(t)$ of the form~\eqref{u_bangbang}.
The main problem addressed in this paper is to find periodic solutions of system~\eqref{cs_affine} under bang-bang control, as formulated below.

\textbf{Periodic control problem.}\footnote{Any solution of $x(t;x^0,u_\tau)$ of the above problem defined on $t\in [0,\tau]$ can, by periodicity be extended to a $\tau$-periodic solution, defined for all $t\in\mathbb R$, under $\tau$-periodic extension the control $u_\tau$.
Therefore, we do not distinguish between solutions on $[0,\tau]$ satisfying the boundary condition $x(0)=x(\tau)$ and their periodic extensions on $\mathbb R$.} {\em
Given $\tau>0$, $N\ge 1$, and a control $u_\tau:[0,\tau]\to U$ of the form~\eqref{u_bangbang}, find an initial condition $x^0\in D$ such that the corresponding solution $x(t;x^0,u_\tau)\in D$ on $t\in [0,\tau]$ satisfies the boundary condition $x(\tau;x^0,u_\tau)=x^0$.}

By introducing the vector fields $f_k:D\to {\mathbb R}^n$,
\begin{equation}\label{f_k}
f_k(x) = f_0(x) + \sum_{j=1}^m u^{(k)}_j g_j(x),\; k=1,2, ... , N,
\end{equation}
we rewrite~\eqref{cs_affine} with controls~\eqref{u_bangbang} in a driftless form:
\begin{equation}\label{cs_driftless}
\dot x(t) = \sum_{k=1}^N \chi_{I_k}(t) f_k(x(t)),\quad t\in [0,\tau], \; x(t)\in D,
\end{equation}
where $\chi_{I_k}:{\mathbb R}\to \{0,1\}$ is the indicator function of the interval $I_k$.
Thus, the above periodic control problem is transformed into finding solutions of the driftless control-affine system~\eqref{cs_driftless} satisfying
$x(0)=x(\tau)$.
We will refer to such solutions as $\tau$-periodic solutions.

To describe the solutions of system~\eqref{cs_driftless}, we denote the (local) flow of a vector field $f_k:D\to{\mathbb R}^n$ by $\Phi_k:{\cal X}\to D$, where ${\cal X}\subset {\mathbb R}\times D$ is an open set containing $\{0\}\times D$. For each $x\in D$, the map $\Phi_k(t,x)$ represents the maximal solution of $\frac{d}{dt} \Phi_k(t,x)=f_k(\Phi_k(t,x))$ such that $\Phi_k(0,x)=x$.
We also adopt the flow notation in terms of the exponential map as $\Phi_k(t,x)=e^{t f_k} x$.
Then, the property that $x(t)$ with $x(0)=x^0$ is  a $\tau$-periodic solution of~\eqref{cs_driftless} is equivalent to requiring that {$x^0\in D$ be a fixed point of the composition}
\begin{equation}\label{flows}
e^{\alpha_1 \tau f_1} e^{\alpha_2\tau f_2} \cdots e^{\alpha_N \tau f_N}.
\end{equation}

\section{MAIN RESULTS}\label{sec_results}
To study the fixed points of the flow composition~\eqref{flows}, we will construct a vector field $F: D\to {\mathbb R}^n$ such that
\begin{equation}\label{flows_Z}
e^{\alpha_1 \tau f_1} e^{\alpha_2\tau f_2} \cdots e^{\alpha_N \tau f_N} = e^F,
\end{equation}
where the concatenation of flows in~\eqref{flows_Z} is read from left to right (cf.~\cite{Duleba}).
For this purpose, we treat $f_1, f_2, ..., f_N$ as elements of a Lie algebra $\mathcal L$ of vector fields on $D$,
and consider their exponential maps $e^{tf_1}$, $e^{tf_2}$, ..., $e^{tf_N}$ in the corresponding Lie group $\mathcal G$.
We recall that, for any $X,Y\in \mathcal L$, the product of their exponentials can be expressed as the exponential of some vector field $Z\in\mathcal L$:
\begin{equation}\label{eXY}
e^X e^Y = e^Z,\; Z = {\mathcal B}(X,Y),
\end{equation}
where ${\mathcal B}(X,Y)$ is given by the Baker--Campbell--Hausdorff--Dynkin (BCHD) formula
as a formal series in iterated commutators of $X$ and $Y$~\cite{strichartz1987campbell,bonfiglioli2011topics}:
{\small
\begin{equation}\label{BCHD}
{\mathcal B}(X,Y)= X + Y + \frac{[X,Y]}{2} +\frac{[X,[X,Y]]-[Y,[X,Y]]}{12}+ ...\, .
\end{equation}}
For vector fields $X$ and $Y$ defined in coordinates as mappings from a column vector in $D\subset {\mathbb R}^n$ to a column vector in ${\mathbb R}^n$, the Lie bracket (commutator) at $x\in D$ is defined as
$
[X,Y](x) = \frac{\partial Y(x)}{\partial x} X(x) -  \frac{\partial X(x)}{\partial x} Y(x)
$,
where $ \frac{\partial X(x)}{\partial x}$ and $\frac{\partial Y(x)}{\partial x}$ denote the Jacobian matrices of $X$ and $Y$, respectively.
Without truncation, the complete infinite series in~\eqref{BCHD} is expressed by Dynkin's notation as follows~\cite{Dynkin1947}:
\begin{equation}\label{Dynkin}
\footnotesize
{\mathcal B}(X,Y)= \sum_{n=1}^\infty \frac{(-1)^{n-1}}{n}\sum_{\substack{r_1+s_1>0,\\ \cdots \\ r_n+s_n>0}} \frac{[X^{r_1} Y^{s_1} X^{r_2} Y^{s_2} \dots X^{r_n}Y^{s_n}]}{\left(\sum_{j=1}^n(r_j+s_j)\right)\cdot \prod_{i=1}^n r_i! s_i!}.
\end{equation}
Here, the iterated Lie brackets are treated as
$$
\scriptsize
\begin{aligned}
[X^{r_1} Y^{s_1}  \dots X^{r_n}Y^{s_n}]=[\underbrace{X,[X,\dots [X}_{r_1\;times},  &[ \underbrace{Y,[Y,\dots [Y}_{s_1\;times}, \dots \\ &[\underbrace{X,[X,\dots [X}_{r_n\;times},  [ \underbrace{Y,[Y,\dots Y}_{s_n\;times}]] \dots ]].
\end{aligned}
$$

By applying this Lie-algebraic approach, we obtain a characterization of $\tau$-periodic trajectories as follows.

\begin{theorem}\label{thm1}
Assume that the vector fields $f_1,f_2,...,f_N:D\to \mathbb R^n$ of system~\eqref{cs_driftless} are analytic in a domain $D$, and let $x^0\in D$.
%
Then, for any sufficiently small $\tau>0$, the following two conditions are equivalent:

\begin{itemize}
\item[(i)] The solution $x(t)$ of the non-autonomous system~\eqref{cs_driftless} for $t\in [0,\tau]$ with the initial data $x(0)=x^0$ satisfies the boundary condition $x(\tau)=x(0)$.
\item[(ii)] The solution $\tilde x(t)$ of the autonomous Cauchy problem
\begin{equation}\label{autonomous}
\begin{aligned}
& \dot {\tilde x}(t) = F (\tilde x(t)),\quad t\in [0,1],\\
& \tilde x(0)=x^0
\end{aligned}
\end{equation}
satisfies the boundary condition $\tilde x(1)=\tilde x (0)$.
\end{itemize}

Here,
\vskip-3ex
\begin{align}
F &(x)= \tau\sum_{i=1}^N \alpha_i f_i(x)
+ \tfrac{\tau^2}{2}\sum_{1\le i<j\le N}\alpha_i\alpha_j[f_i,f_j](x) \nonumber \\
&+ \frac{\tau^3}{12}\sum_{\substack{1\le i,j\le N\\ i\ne j}} \alpha_i^2 \alpha_j [f_i,[f_i,f_j]](x) \nonumber \\
&
- \frac{\tau^3}{12}\sum_{1\le i<j\le N} \alpha_i\alpha_j^2 [f_j,[f_i,f_j]](x) \nonumber \\
&- \frac{\tau^3}{4}\sum_{1\le i<j<k\le N} \alpha_i\alpha_j \alpha_k [f_k,[f_i,f_j]](x)
+ \mathcal{O}(\tau^4).\label{F_general}
\end{align}
\end{theorem}
\vskip1ex

\begin{proof}
In the case $N=1$, formula~\eqref{F_general} reduces to $F(x)=\tau f_1(x)$,
so that for each solution $x(t)$ of system~\eqref{cs_driftless}, the corresponding function $\tilde x(t)=x(\tau t)$ satisfies system~\eqref{autonomous}.
This establishes the equivalence of~(i) and~(ii).

For $N\ge 2$, we apply the BCHD formula to derive the vector field $F:D\to{\mathbb R}^n$ satisfying equation~\eqref{flows_Z}.
If $N=2$, we obtain $F=F_2$ from representations~\eqref{eXY} and~\eqref{Dynkin}, where
\begin{equation}\label{F2}
\small
\begin{aligned}
&F_2 = {\cal B}(\alpha_1\tau f_1,\alpha_2\tau f_2) = \tau(\alpha_1 f_1 + \alpha_2 f_2) + \frac{\alpha_1 \alpha_2 \tau^2  }{2} [ f_1, f_2]\\ &+\frac{\alpha_1\alpha_2\tau^3 }{12}(\alpha_1 [ f_1,[f_1, f_2]]
-\alpha_2 [ f_2,[f_1,f_2]]) + \mathcal{O}(\tau^4).
\end{aligned}
\end{equation}
Then, for arbitrary $N>2$, we recursively define
\begin{equation}\label{Fk}\
F_k = {\cal B}(F_{k-1},\alpha_k\tau f_k),\quad k=3,..., N.
\end{equation}
Applying the CBHD formula~\eqref{Dynkin} to~\eqref{Fk} yields
$$
F_k = F_{k-1} + \alpha_k\tau f_k + \frac{\alpha_k\tau}{2} [F_{k-1}, f_k]
+\frac{\alpha_k\tau}{12}[F_{k-1},[F_{k-1}, f_k]]
-\frac{\alpha_k^2\tau^2}{12} [f_k,[F_{k-1}, f_k]] + \mathcal{O}(\tau^4)$$ for $k=3,..., N$.
By collecting terms according to powers of $\tau$ in $F_N$, we conclude that expression~\eqref{F_general} coincides with the recursively defined $F_N$ for $N>3$,
and with $F_2$ in~\eqref{F2} for $N=2$.

Since the vector fields $f_1$, ..., $f_N$ are analytic in $D$,
the CBHD expansions~\eqref{F2}--\eqref{Fk} converge for all sufficiently small $\tau>0$ (see, e.g.,~\cite[Chap.~5]{bonfiglioli2011topics} for convergence conditions of the CBHD formula).
For such $\tau>0$, $\tilde x(1)=e^{F}x^0$ represents the solution of~\eqref{autonomous} at $t=1$, and, due to~\eqref{flows_Z},
\begin{equation}\label{x_equivalence}
\tilde x(1)= x(\tau),
\end{equation}
where $x(t)$ is the solution of~\eqref{cs_driftless} with  $x(0)=x^0$.
Therefore, the equivalence of~(i) and~(ii) follows from~\eqref{x_equivalence}.
\end{proof}

\begin{remark}\label{rem1}
As system~\eqref{cs_driftless} is obtained from~\eqref{cs_affine} through the transformations of vector fields in~\eqref{f_k},
Theorem~\ref{thm1} ensures the solvability conditions of the periodic control problem formulated in Section~\ref{sec_problem}, with the vector fields in~\eqref{F_general} expressed in terms of $f_0$, $g_1$, ..., $g_m$.
Although the formula for $F$ in~\eqref{F_general} is explicitly presented up to terms of order ${\mathcal O}(\tau^3)$, the proof of Theorem~\ref{thm1} provides an inductive procedure for defining $F$ with arbitrary accuracy.
Higher-order iterated Lie brackets in Dynkin's formula~\eqref{Dynkin} can, in principle, be taken into account using a computer algebra system.
{
We implemented freely available code for \texttt{Maple 2024} that computes the truncated expansion of $F$ for arbitrary numbers $N$ of vector fields on $\mathbb R^n$ and maximal iterated Lie bracket length $M$.
The code is available in the GitLab repository~\cite{BCHD_Code}, although practical computations for large $n$, $N$, and $M$ may be constrained by available computational resources.}
\end{remark}

Although a test for periodic solutions of the nonlinear autonomous system~\eqref{autonomous} remains a challenging problem in general and can be approached using the Poincar\'e map~\cite[Chapter~6]{robinson2012introduction}, contraction analysis~\cite{lohmiller1998contraction,tsukamoto2021contraction}, Hopf bifurcation theory, and normal forms~\cite{han2012normal,grushkovskaya2013asymptotic},
we outline below a particular corollary of Theorem~\ref{thm1}
{that follows from the Bendixson–Dulac theorem for planar systems~\cite[Section~3.9]{perko2013differential}}.

{
\begin{corollary}\label{cor1}
Assume that $n=2$, and let the vector fields $f_1, ..., f_N: D\to {\mathbb R}^2$ be analytic in a simply connected domain $D\subset {\mathbb R}^2$.
Let $\tau>0$ be sufficiently small so that the series of nested commutators defining the vector field $F:D\to {\mathbb R}^2$ in Theorem~\ref{thm1} converges.
Suppose, moreover, that there exists a function $\rho\in C^1(D;{\mathbb R})$ such that
\begin{equation}\label{divcond}
\textrm{div} \bigl(\rho(x) F(x)\bigr)  >0\;\text{for almost all}\; x\in D.
\end{equation}
Then, the existence of a solution $x(t)\in D$ of~\eqref{cs_driftless} satisfying $x(0)=x(\tau)$ is equivalent to the condition $F(x(0))=0$.
\end{corollary}}

{
For an arbitrary state space dimension $n$, we further characterize conditions under which the study of periodic trajectories of~\eqref{autonomous} can be reduced to the analysis of its equilibria by extending the ideas of contraction analysis~\cite{lohmiller1998contraction,lohmiller2000nonlinear} and their generalizations in the sense of incremental stability~\cite{angeli2002lyapunov,forni2013differential}.}

{Following the incremental stability approach on manifolds~\cite{forni2013differential}, we consider system~\eqref{autonomous} together with its variational dynamics as:
\begin{eqnarray}\label{xFx}
&\dot x=F(x),\\
&\dot{\delta x}
=
\frac{\partial F(x)}{\partial x}\delta x.
\label{eq:prolonged}
\end{eqnarray}
The prolonged system~\eqref{xFx}--\eqref{eq:prolonged} evolves on the tangent bundle $TD\simeq D\times\mathbb R^n\subset\mathbb R^{2n}$.
Let $\|\delta x\|$ denote the Euclidean norm of $\delta x\in {\mathbb R}^n$,
then a candidate Finsler--Lyapunov function for~\eqref{xFx}--\eqref{eq:prolonged}, in the sense of~\cite{forni2013differential}, is a function $V:D\times{\mathbb R^n}\to {\mathbb R}_{\ge 0}$ satisfying
\begin{equation}\label{Vfun}
b_1 \|\delta x\|^p \le V(x,\delta x) \le b_2 \|\delta x\|^p,\;\forall (x,\delta x)\in D\times\mathbb R^n,
\end{equation}
where $b_1>0$, $b_2>0$, and $p\ge 1$ are constants.
}

{By Proposition~2 of~\cite{forni2013differential}, if $D_0\subset D$ is an invariant set for~\eqref{xFx} and there exists a Lyapunov function
satisfying
\begin{equation}
\dot V_{\eqref{xFx}}(x,\delta x):=\frac{\partial V}{\partial x}F(x)
+
\frac{\partial V}{\partial \delta x}
\frac{\partial F}{\partial x}\delta x
\le
-\theta(V(x,\delta x)),
\label{V-decay}
\end{equation}
for all $(x,\delta x)\in D_0\times {\mathbb R^n}$,
then system~\eqref{xFx} admits no nonconstant periodic trajectories in $D_0$. Combining this result with Theorem~\ref{thm1} yields the following important corollary.}

{
\begin{corollary}\label{cor2}
Let the assumptions of Theorem~\ref{thm1} hold, and let $D_0\subset D$ be a forward invariant set for~\eqref{xFx}, where the vector field $F$  defined in Theorem~\ref{thm1}.
If there exist functions $V\in C^1(D\times{\mathbb R^n}; {\mathbb R}_{\ge 0})$, $\theta \in\mathcal K$ and constants $b_1,b_2>0$, $p\ge 1$
such that inequalities~\eqref{Vfun}--\eqref{V-decay} hold, then every $\tau$-periodic solution $x(t)\in D_0$ of~\eqref{cs_driftless} satisfies
\begin{equation}\label{equil}
F(x(0))=0.
\end{equation}
\end{corollary}}

Under additional assumptions, we prove the uniqueness of {$\tau$-periodic solutions of system~\eqref{cs_affine} for any admissible periodic control.
Let $u^*\in L^\infty\left([0,+\infty);U\right)$ be $\tau$-periodic in the sense that $u^*(t+\tau)=u^*(t)$ for almost all $t\ge 0$, then we consider system~\eqref{cs_affine} with $u=u^*(t)$:
\begin{equation}\label{sys_discontinuous}
\dot x = f(t,x),\quad f(t,x):= f_0(x) + \sum_{j=1}^m u^*_j(t) g_j(x).
\end{equation}
If the vector fields $f_0$, $g_1$, ..., $g_m$ are of class $C^1(D)$,
then for any $t_0\ge 0$ and $x^0\in D$, there exists a unique maximal Carathéodory solution $x(t)=\varphi_t(t_0,x^0)$ of~\eqref{sys_discontinuous} satisfying $x(t_0)=x^0$~\cite[Chap.~1]{Filippov},
and we focus on describing the properties of the non-autonomous flow
 $\varphi_t$ in terms of incremental stability.
For this purpose, we define the distance between any two points $x^1, x^2\in D$ by
$d(x^1,x^2):=\inf_{\gamma\in\Gamma(x^1,x^2)} \int_0^1 \|\dot \gamma (s)\|ds$, where
$\Gamma(x^1,x^2) :=  \left\{ \gamma\in C^1([0,1];D)\;|\; \gamma(0)=x^1,\; \gamma (1) = x^2\right\}$.
This definition allows us to consider distances in arbitrary nonconvex domains, and $d(x^1,x^2)=\|x^1-x^2\|$ whenever $D$ is convex.
According to~\cite{forni2013differential}, system~\eqref{sys_discontinuous} is {\em incrementally asymptotically stable (IAS)} on a
forward invariant set $D_0\subset D$ if there exists a function $c\in\mathcal K$ such that, for all $x^1,x^2\in D_0$ and all $t\ge t_0\ge 0$,
\begin{eqnarray}\label{IAS}
&d(\varphi_t(t_0,x^1),\varphi_t(t_0,x^2))\le c (d(x^1,x^2)),\\
&\lim_{t\to+\infty }d(\varphi_t(t_0,x^1),\varphi_t(t_0,x^2)) =0.\label{IAS_conv}
\end{eqnarray}
}
{
\begin{theorem}\label{thm2}
Assume that $f_0,g_1,...,g_m\in C^2(D)$,
$u^*\in L^\infty \left([0,+\infty);U\right)$ is a $\tau$-periodic control,
and that there exist functions $V\in C^1(D\times{\mathbb R^n}; {\mathbb R}_{\ge 0}) $, $\theta\in\mathcal K$ satisfying~\eqref{Vfun} and
\begin{equation}\label{Vdot}
\dot V_{\eqref{sys_discontinuous}}(x,\delta x,t) :=
\frac{\partial V}{\partial x}f
+
\frac{\partial V}{\partial \delta x}
\frac{\partial f}{\partial x}\delta x
\le
-\theta(V(x,\delta x)),
\end{equation}
for all $(x,\delta x,t)\in D\times{\mathbb R}^n\times {\mathbb R}_{\ge 0}$.
Let $D_0\subset D$ be a closed domain, and let $x:[0,+\infty)\to D_0$ be any Carathéodory solution of~\eqref{sys_discontinuous}.
Then $x(t)$ converges to the unique $\tau$-periodic solution $x^*(t)$ of~\eqref{sys_discontinuous} in $D_0$ as $t\to+\infty$.
\end{theorem}}

\begin{proof}
{
For any $x^1,x^2\in D$ and a $C^1$-curve $\gamma\in \Gamma(x^1,x^2)$,
consider the function $$v(t,s):=V(\varphi_t(t_0,\gamma(s)),\frac{\partial}{\partial s}\varphi_t(t_0,\gamma(s))).$$
For any fixed $s\in [0,1]$, the mapping $t\mapsto v(t,s)$ belongs to $AC({\mathcal I})$, where $\mathcal I$ is the interval on which the flow is defined,
since the solutions are of Carathéodory type~\cite[Chap.~1]{Filippov}.
We then evaluate the time-derivative of $v$ for almost all $t\in \mathcal I$:
\begin{equation}\label{vt}
\begin{aligned}
&\frac{\partial v(t,s)}{\partial t} = \frac{\partial V}{\partial x} f(t,\varphi_t(t_0,\gamma(s))) + \frac{\partial V}{\partial \delta x} \frac{\partial^2}{\partial t\partial s}\varphi_t(t_0,\gamma(s))\\
& = \dot V_{\eqref{sys_discontinuous}}(\varphi_t(t_0,\gamma(s)),\frac{\partial}{\partial s}\varphi_t(t_0,\gamma(s)),t)\le -\theta(v(t,s)).
\end{aligned}
\end{equation}
Here, the partial derivatives of $V$ are evaluated at $x=\varphi_t(t_0,\gamma(s))$ and $\delta x = \frac{\partial}{\partial s}\varphi_t(t_0,\gamma(s))$,
and the last inequality follows from~\eqref{Vdot}.
Integrating~\eqref{vt} yields $v(t,s)\le v(t_0,s)$ for all $t\in \mathcal I\cap [t_0,+\infty)$.
As in the proof of~\cite[Thm.~1]{forni2013differential}, we estimate $d(\varphi_t(t_0,x^1),\varphi_t(t_0,x^2))$ in terms of
$\int_0^1 v(t,s)ds$ under assumption~\eqref{Vfun}, provided that $\gamma$ is chosen as an appropriate $\epsilon$-minimizer.
This implies~\eqref{IAS} for all $t\ge t_0$ for which the corresponding solutions are defined,
and also~\eqref{IAS_conv} if both $\varphi_t(t_0,x^j)$ are defined for all $t\in [t_0,+\infty)$.}

{Let $x(t)\in D_0$ be a Carathéodory solution of system~\eqref{sys_discontinuous} for $t\ge 0$}. Consider the sequence of its shifts { -- solutions $x^{(k)}(t) = x(t + k\tau)\in D_0$}, defined for all $k\in {\mathbb N}_0$ and $t\ge 0$.
{By~\eqref{IAS_conv}, applied to the shifted solutions $x^{(k)}$ and $x^{(l)}$, the sequence $\{x(t+k\tau)\}_{k\in{\mathbb N}_0}$ is Cauchy
for every fixed $t\ge 0$.
Since $D_0$ is closed and therefore complete, the limit}
\begin{equation}\label{x_lim}
x^*(t) := \lim_{k\to+\infty} x^{(k)}(t)
\end{equation}
{
exists and belongs to $D_0$.
Furthermore, $x^*(t)$ is $\tau$-periodic because of its construction in~\eqref{x_lim},
and $x^*(t)$ is a Carathéodory solution of~\eqref{sys_discontinuous} as the limit of Carathéodory solutions~\cite[Lemma~3]{Filippov}.
Moreover,~\eqref{IAS_conv} implies $\|x(t)-x^*(t)\|\to 0$ as $t\to+\infty$.
Finally, $x^*(t)$ is the unique periodic solution of~\eqref{sys_discontinuous} in $D_0$;
otherwise,~\eqref{IAS_conv} would be violated.}
\end{proof}



\begin{remark}\label{rem3}{
A computationally attractive design of candidate Lyapunov functions for Corollary~\ref{cor2} and Theorem~\ref{thm2} is based on quadratic forms (cf.~\cite{forni2013differential}):
\begin{equation}\label{Vquad}
V(x,\delta x)= {\delta x}^\top Q \delta x,\; Q=Q^\top \succ 0.
\end{equation}
By evaluating $\dot V$ and considering linear functions $\theta\in \cal K$, we conclude that~\eqref{Vquad} satisfies~\eqref{Vfun}--\eqref{V-decay} if
a constant $n\times n$ matrix $Q\succ 0$ and a scalar $\beta>0$ satisfy the following Lyapunov-type linear matrix inequality (LMI):
\begin{equation}\label{Lyapunov_LMI}
Q\frac{\partial F(x)}{\partial x} + \left( \frac{\partial F(x)}{\partial x} \right)^\top  Q \preceq -\beta I,
\end{equation}
for all $x\in D_0$,
where $I$ is the identity matrix,
the relations ``$\succ 0$'' and ``$\preceq 0$'' denote positive-definite and negative-semidefinite matrices, respectively.
The same LMI~\eqref{Lyapunov_LMI}, with the Jacobian of $F(x)$ replaced by $\frac{\partial f(t,x)}{\partial x}$ for all $(t,x)\in {\mathbb R}_{\ge 0} \times D$, can be used to ensure condition~\eqref{Vdot} of Theorem~\eqref{thm2}. In the particular case where~\eqref{sys_discontinuous} coincides with  system~\eqref{cs_driftless} and the vector fields $g_1$, ..., $g_m$ are constant, we have $\frac{\partial f(t,x)}{\partial x} \equiv \frac{\partial f_0(x)}{\partial x}$ for all $t\ge 0$.
Hence, the conditions of Theorem~\ref{thm2} are guaranteed in this case by
\begin{equation}\label{LMI_f0}
Q\frac{\partial f_0(x)}{\partial x} + \left( \frac{\partial f_0(x)}{\partial x} \right)^\top  Q \preceq -\beta I.
\end{equation}}
In the considered case, condition~\eqref{LMI_f0} implies that the ${\cal O}(\tau)$ Taylor approximation of the vector field $F$ in~\eqref{F_general} has a uniformly Hurwitz Jacobian in $D_0$. Therefore, if $F^{(1)}(x)=\alpha_1 f_1(x) + ... + \alpha_N f_N(x)$ vanishes at some interior point $x^0$ of $D_0$, then the equilibrium $\tilde x= x^0$ of the autonomous system $\dot {\tilde x}(t) = \tau F^{(1)}(\tilde x(t))$ is exponentially stable in~$D_0$~\cite[Thm.~2]{lohmiller1998contraction}.
Although a complete stability analysis of the autonomous system in~\eqref{autonomous} from Theorem~\ref{thm1} requires considering all higher-order terms in the vector field~$F$, the above observation suggests {testing Corollary~\ref{cor2} and Theorem~\ref{thm2} with a quadratic Lyapunov candidate  obtained from~\eqref{LMI_f0}.}
\end{remark}

\section{CASE STUDIES}\label{sec_example}
\subsection{Reaction of the type ``$A\to$ Product'' }
In this first example, we apply the proposed description of periodic trajectories to a mathematical model of a nonisothermal chemical reaction.
A reaction of the type ``$A\to$ Product'' of order $\bar n$, considered in~\cite{CES2017,AMM2019}, is described by the following nonlinear control-affine system:
\begin{equation}\label{CSTR_affine}
\small
\dot x = f_0(x) + u_1 g_1(x) + u_2 g_2(x),\,x=\begin{pmatrix}x_1 \\ x_2\end{pmatrix}\in D, u=\begin{pmatrix}u_1 \\ u_2\end{pmatrix}\in U,
\end{equation}
where $
D = (-1,+\infty)^2
$,
$
U = [u_1^{\min},u_1^{\max}] \times  [u_2^{\min},u_2^{\max}]
$,
$$
\begin{aligned}
f_0(x)&=\begin{pmatrix}
- \phi_1 x_1 + k_1 e^{-\varkappa} - (x_1+1)^{\bar n} e^{-\varkappa/(x_2+1)}\\
- \phi_2 x_2 + k_2 e^{-\varkappa} - (x_1+1)^{\bar n} e^{-\varkappa/(x_2+1)}
\end{pmatrix},\\
g_1(x)&=(1, 0)^\top,\; g_2(x)=(0,1)^\top.
\end{aligned}
$$
In this control system, $x_1$ and $x_2$ describe deviations of the outlet concentration of $A$ and the outlet temperature, respectively, from their steady-state operation values
in a continuously stirred tank reactor (CSTR).
The controls $u_1$ and $u_2$ denote deviations of the inlet concentration of $A$ and the inlet temperature from their steady-state values, respectively.
The state variables, control inputs, and time in~\eqref{CSTR_affine} are normalized to be dimensionless, and the parameters of the vector field $f_0$ are also dimensionless (see~\cite{CES2017} for details on the derivation of equations~\eqref{CSTR_affine}).
The constraint $x_1>-1$ ensures the positivity of the physical concentration, while $x_2>-1$ ensures the positivity of the absolute temperature.
We choose the parameters as in~\cite{AMM2019}:
${\bar n}=1$, $\phi_1=\phi_2=1$,
$k_1=5.819\cdot 10^7$, $k_2=-8.99\cdot 10^5$, $\varkappa=17.77$, $u_1^{\max}=-u_1^{\min} = 1.798$, $u_2^{\max}=-u_2^{\min}=0.06663$.


\begin{figure}[t]\label{BCH4}
    \begin{subfigure}[b]{0.48\columnwidth}
        \includegraphics[width=0.9\columnwidth]{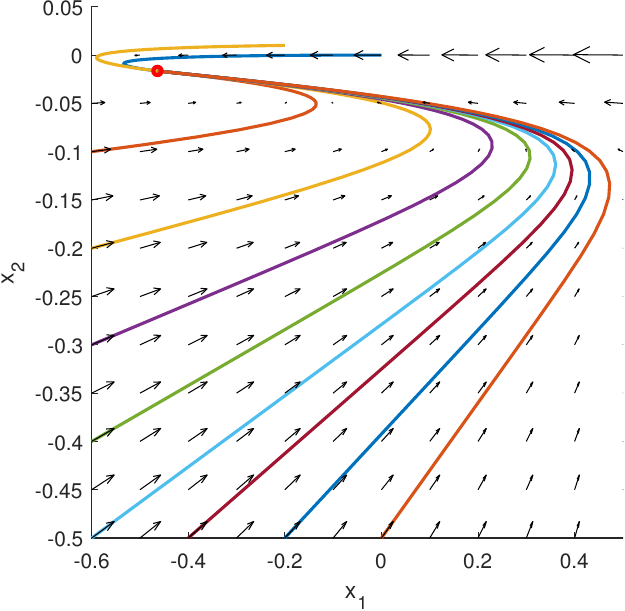}
        \caption{Phase portrait of~\eqref{FM_system}}
        \label{fig:right}
    \end{subfigure}
       \hfill
        \begin{subfigure}[b]{0.5\columnwidth}
            \includegraphics[width=\columnwidth]{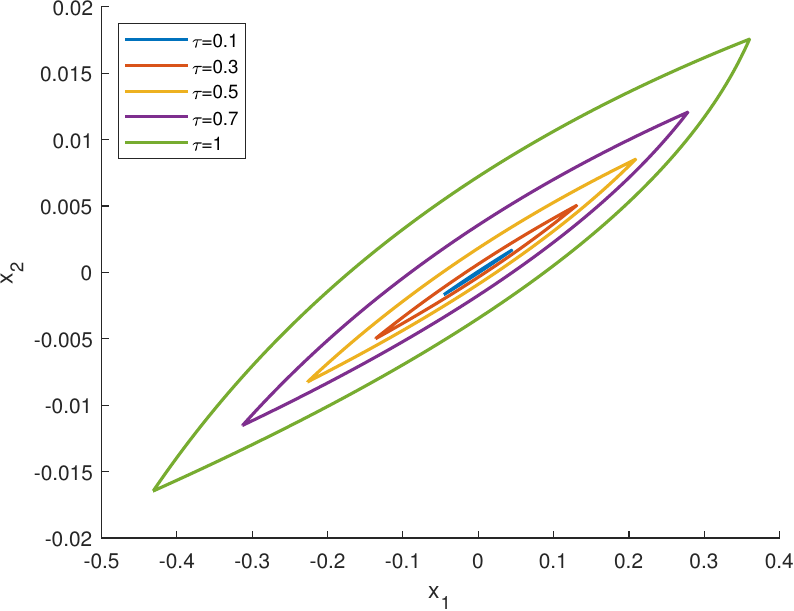}
        \caption{Trajectories of~\eqref{CSTR_affine}}
        \label{fig:left}
    \end{subfigure}
\caption{Trajectories of systems~\eqref{CSTR_affine} and~\eqref{FM_system}, $N=2$.}
\end{figure}

The case $N=2$ is reported in~\cite{AMM2019} to be locally optimal for the optimization of the performance of the considered chemical reaction model.
Therefore, we focus on system~\eqref{cs_driftless} with the vector fields $f_1, f_2: D\to{\mathbb R}^2$ defined by~\eqref{f_k} with $N=2$, i.e.
$
f_1 = f_0 + u_1^{(1)} g_1 + u_2^{(1)} g_2$,
$
f_2 = f_0 + u_1^{(2)} g_1 + u_2^{(2)} g_2
$,
where the control values $u^{(1)}$ and $u_2^{(2)}$ belong to the boundary $\partial U$ of $U$.
For the given vector fields of~\eqref{CSTR_affine}, the corresponding vector field $
F = {\tilde F}_M + {\mathcal O}(\tau^{M+1})
$
in Theorem~\ref{thm1} is derived using \texttt{Maple} {\cite{BCHD_Code},
and the following ``BCHD-approximate system'' is constructed:
\begin{equation}\label{FM_system}
\dot x = {\tilde F}_M ( x) =\sum_{k=1}^M \tau^k F^{(k)}(x),\; x\in D.
\end{equation}
}
According to Theorem~\ref{thm1}, the $\tau$-periodic solutions of system~\eqref{CSTR_affine} with bang-bang controls $u=u_\tau(t)$ of the form~\eqref{u_bangbang}
are determined by $1$-periodic solutions of~\eqref{autonomous}.
Under the above choice of kinetic parameters and control parameters
\begin{equation}\label{control_parameters}
\alpha_1 = \alpha_2 = \frac12,\; u^{(1)}_j = u_j^{\max},\; u^{(2)}_j = u_j^{\min},\; j=1,2,
\end{equation}
the phase portrait of system~\eqref{FM_system} has been analyzed numerically for $M=4$ using \texttt{MATLAB R2025a}.
We observe that system~\eqref{FM_system} possesses an attracting point
$x^*\in D$, which is depicted by a red circle in Fig.~1(a).

The considered solutions $\tilde x(t)$ of system~\eqref{FM_system} converge to the attractor $x^*$ for large $t$, and the coordinates of $x^*$ are evaluated numerically as zeros of ${\tilde F}_M$ for different $M$.
These computations are summarized in Table~1 for $\tau=1$ and $M\le 4$.
\begin{table}[h]\label{tab1}
\begin{center}
\renewcommand{\arraystretch}{1.4}
\begin{tabular}{|c|m{4cm}|c|}
\hline
$M$ &  $F^{(k)}$ in~\eqref{FM_system} & $(x^*_1,x^*_2)$ \\[2pt] \hline
1 & $\alpha_1 f_1+\alpha_2 f_2$ &$(0, 0)$ \\[2pt] \hline
2 & $\frac{\alpha_1\alpha_2}{2}[f_1,f_2]$& $(-0.3651, -0.01796)$\\[2pt] \hline
3 & $\frac{\alpha_1 \alpha_2 ( \alpha_1[f_1,[f_1,f_2]] - \alpha_2 [f_2,[f_1,f_2]])}{12}$ &$(-0.4638, -0.01644)$ \\[2pt] \hline
4 & $- \frac{\alpha_1^2 \alpha_2^2}{24}{} [f_2,[f_1,[f_1,f_2]]]$ & $(-0.4384, -0.01634)$ \\[2pt] \hline
\end{tabular}\\
\vskip0.5em
Table~1. Equilibrium $x^*=(x^*_1,x^*_2)^\top$ of system~\eqref{FM_system} for different $M$.
\end{center}
\end{table}
We have also tested the sign of the divergence of ${\tilde F}_M(x)$ numerically in~\eqref{FM_system} and observed that
$
\textrm{div}\, {\tilde F}_M(x) <0
$
for all $x\in D$: $|x_1|<0.999$, $|x_2|<0.999$ and all $M=2,3,4$.
Thus, according to {Corollary~\ref{cor1} with $\rho(x)\equiv -1$},
the characterization of periodic trajectories of~\eqref{CSTR_affine} in terms of the equilibrium of~\eqref{FM_system} is well confirmed in our case study.

To compare these results with other approaches, we compute the initial data $x^*$ for the $\tau$-periodic solution of system~\eqref{CSTR_affine} with controls~\eqref{u_bangbang},~\eqref{control_parameters},
obtained after $10$ iterations of the modified Newton's method from~\cite{IEEE2025} with $N=2$.
As a result, we obtain:
$
x^*\approx (-0.4314,-0.01646)^\top
$.
By comparing this $x^*$ with the rows of Table~1, we observe that the equilibrium of system~\eqref{FM_system} for $M=4$ already provides an acceptable approximation of the initial data for the corresponding periodic solution of system~\eqref{CSTR_affine}.
Periodic trajectories of system~\eqref{CSTR_affine} with controls~\eqref{u_bangbang}, corresponding to the periods $\tau=0.1$, $0.3$, $0.5$, $0.7$, $1$, are shown in Fig.~1(b).

{
To illustrate the computational consistency of the proposed approach for a higher number of switchings,
Fig.~2(a) shows the trajectories of~\eqref{CSTR_affine} with periods $\tau=0.1, 0.5, 1$ and controls~\eqref{u_bangbang} with $N=3$ and $\alpha_1=\alpha_2=\alpha_3=1/3$,
where the initial value $x^0$ is computed as the zero of the truncated expansion of $F$ up to order $M=4$ from Theorem~\ref{thm1}.}


\subsection{Reaction of the type ``$A\to B\to Product$''}
A controlled nonisothermal chemical reaction of the type ``$A\to B\to Product$'' is described by (cf.~\cite{fogler2022elements}):
\begin{equation}\label{CSTR3_affine}
\dot x = f_0(x) + u_1 g_1(x) + u_2 g_2(x),\; x\in {\mathbb R}^3, u\in {\mathbb R}^2,
\end{equation}
where $x_1$ and $x_2$ denote the concentrations of $A$ and $B$, respectively,
$x_3$ is the temperature, and the vector fields are $g_1(x) = \left({F}/{V},0,0\right)^\top$, $g_2(x) = \left(0,0,{F}/{V} \right)^\top$,
$$
f_0(x) =  - \begin{pmatrix}
\left( \frac{F}{V} + k_1^0 e^{-\frac{E_1}{Rx_3}} \right) x_1\\
(\frac{F}{V} + k_2^0 e^{-\frac{E_2}{Rx_3}})x_2 - k_1^0 x_1 e^{-\frac{E_1}{Rx_3}} \\
\frac{F x_3}{V}  +  \frac{\Delta H_1 k_1^0 x_1}{\rho C_p} e^{-\frac{E_1}{Rx_3}} + \frac{\Delta H_2 k_2^0 x_2}{\rho C_p} e^{-\frac{E_2}{Rx_3}}
\end{pmatrix}.
$$
The reaction is assumed to be adiabatic and is controlled by $u_1$ (the inlet concentration of $A$) and $u_2$ (the inlet temperature) in a neighborhood of their reference values $\bar u_1$ and $\bar u_2$.
We choose the following realistic parameters for this system~(cf.~\cite[Example~2.5]{seborg2016process}):
$F=100$,
$V=100$, $R = 8.314$,
$k_1^0 = 7.2\cdot 10^{10}$,
$k_2^0 = 10^{10}$,
$E_1 = 7.275\cdot 10^4$,
$E_2 = 8\cdot 10^4$,
$\Delta H_1 = - 5\cdot 10^4$,
$\Delta H_2 = - 7\cdot 10^4$,
$\rho C_p = 4.2\cdot 10^3$,
$\bar u_1 = 1$,
$\bar u_2 = 350$ {(physical dimensions are omitted for simplicity)}.
For the steady-state controls $u_1=\bar u_1$ and $u_2 = \bar u_2$, system~\eqref{CSTR3_affine} has an equilibrium $\bar x\approx (0.3683, 0.6189, 357.7354)^\top$.

As in the previous example, we consider the switching control~\eqref{u_bangbang} with $N=2$ and the switching scenario~\eqref{control_parameters} {with
$u_1^{\min} = 0.5$, $u_1^{\max} = 1.5$, $u_2^{\min} = 300$, $u_2^{\max} = 400$}.
Under this choice of parameters and the time horizon $\tau=1$, we numerically determine the equilibrium ${\tilde x}^*$ of system~\eqref{autonomous},
where the vector field $F$ in~\eqref{autonomous} is approximated up to terms of order $M=4$. Using the \texttt{fsolve} function in \texttt{MATLAB}, we obtain
$
{\tilde x}^* \approx \left(
   0.2582478,
   0.6062874,
   357.4668
   \right)^\top
$.
The components of the corresponding periodic solution $x(t)$ of~\eqref{CSTR3_affine} with the initial data $x(0)={\tilde x}^*$ are shown in Fig.~2(b).

\begin{figure}[t]\label{BCH4}
    \begin{subfigure}[b]{0.48\columnwidth}
        \includegraphics[width=0.9\columnwidth]{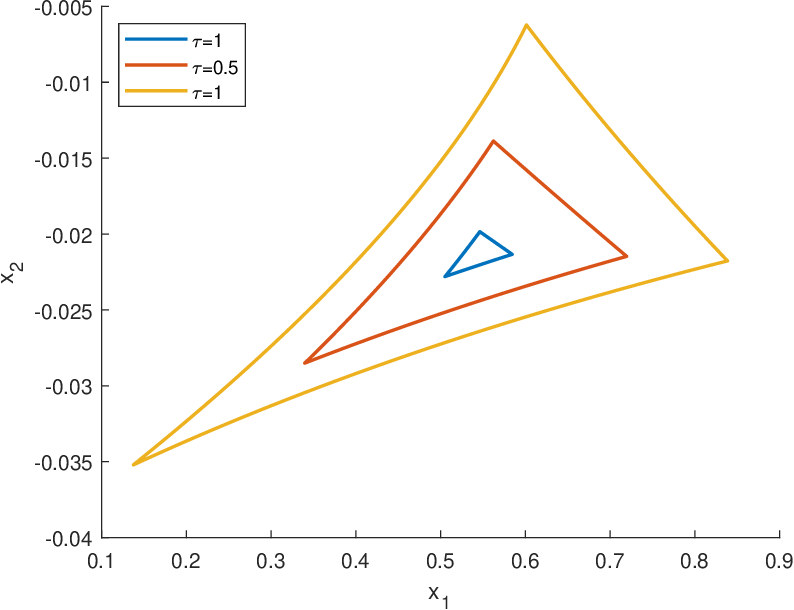}
        \caption{Trajectories of~\eqref{CSTR_affine}: $N=3$}
        \label{fig:right}
    \end{subfigure}
       \hfill
        \begin{subfigure}[b]{0.5\columnwidth}
            \includegraphics[width=0.95\linewidth]{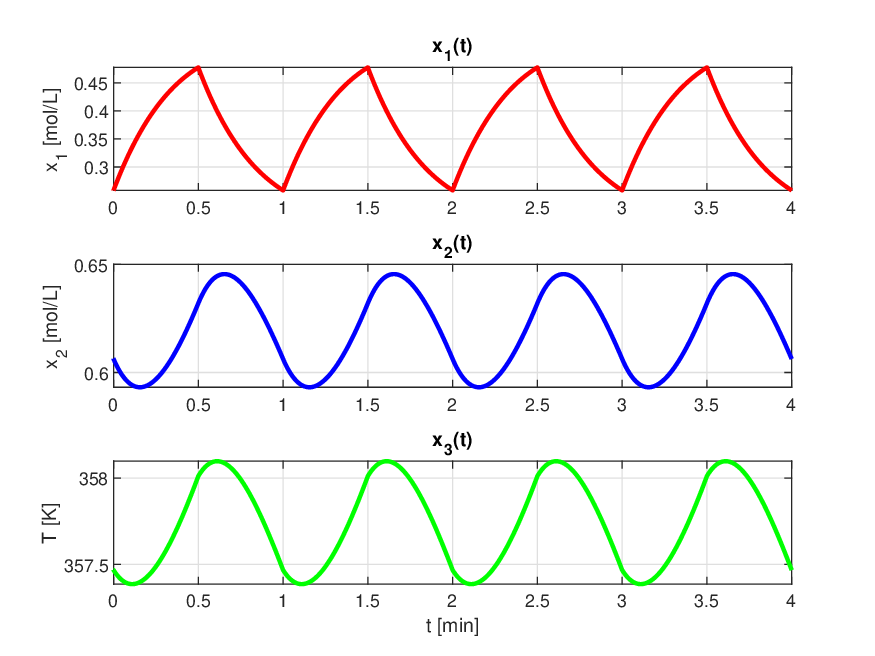}
        \caption{Periodic solution of~\eqref{CSTR3_affine}}
        \label{fig:left}
    \end{subfigure}
\caption{Periodic trajectories: $N=3$ and $n=3$.}
\end{figure}

Note that the matrix $J=\left.\frac{\partial f_0(x)}{\partial x}\right|_{x=\bar x}$ is Hurwitz,
and the solution of the Lyapunov equation $Q A + A^\top Q = - I$ is
$$
Q\approx \begin{pmatrix}
   32.1045& 1.3812 & 4.1283 \\
   1.3812 & 0.5365 & 0.1375 \\
   4.1283 & 0.1375 & 0.6974
\end{pmatrix}.
$$

Numerical computations in~\texttt{MATLAB} show that {both LMIs~\eqref{LMI_f0} and~\eqref{Lyapunov_LMI} with the truncation $\tilde F_M$} hold in $D_0=\{x\in{\mathbb R}^3|\,\, |x_i-{\bar x_i}| \le \delta_i \bar x_i,\; i=1,2,3\}$ with $\delta_1=0.3$, $\delta_2=0.65$, $\delta_3=0.999$, and $\beta=0.1$.
{
In view of Remark~\ref{rem3}, this verifies the assumptions of Theorem~\ref{thm1} and Corollary~\ref{rem3},
which computationally reduce the periodic control problem to finding the zeros of the truncated $\tilde F_M$ $(M=4)$.
Theorem~\ref{thm2} then establishes the stability of the corresponding periodic solution; Fig.~2(b) depicts one such solution.}
%


\section{CONCLUSIONS}\label{sec_conc}

Although {Theorem~\ref{thm1} links the {\em time-varying} periodic control problem to localizing periodic trajectories of the highly nonlinear {\em autonomous} system~\eqref{autonomous}, this framework can be substantially simplified under the assumptions of Corollaries~\ref{cor1} and~\ref{cor2}, reducing the problem to finding equilibria of the vector field $F$. The computational results support these assumptions. In both reaction models considered, the corresponding autonomous system admits an equilibrium $x^*$, whose coordinates can be efficiently computed numerically {\em without solving differential equations}.
Stability is established by Theorem~\ref{thm2} via incremental stability arguments.
Although the BCHD approach is theoretically justified only for sufficiently small periods $\tau$,
the numerical results in Section~\ref{sec_example} show its effectiveness for a ``moderate'' period $\tau=1$.}

\section*{ACKNOWLEDGMENT}
The authors gratefully acknowledge the funding by the European
Regional Development Fund (ERDF) within the programme Research and
Innovation -- Grant Number ZS/2023/12/182138.

\end{document}